\documentclass[journal,twoside,web]{ieeecolor}
\usepackage{lcsys}
\usepackage{silence}
\usepackage[noadjust]{cite}
\usepackage{amsmath,amssymb,amsfonts}
\usepackage{algorithmic}
\usepackage{algorithm}
\usepackage{graphicx}
\usepackage{graphbox}
\usepackage{textcomp}
\usepackage{mathtools}
\usepackage{multirow}
\usepackage{balance}
\usepackage[export]{adjustbox}
\usepackage{color}
\usepackage{subcaption}
\usepackage{accents}
\usepackage{textcomp}
\usepackage{comment}
\usepackage{bm}
\usepackage{tikz}
\usepackage{upgreek}
\let\labelindent\relax
\usepackage[shortlabels]{enumitem}
\usepackage{xcolor}

\newcommand{\blue}[1]{\color{black}{#1}\color{black}\phantom{}}
\renewcommand{\det}{\text{ \normalfont{det}}}

\newtheorem{theorem}{Theorem}

\newtheorem{lemma}{Lemma}

\newtheorem{proposition}{Proposition}

\makeatletter
\let\NAT@parse\undefined
\makeatother
\usepackage[hidelinks]{hyperref}

\title{Accelerated Alternating Direction Method of Multipliers  Gradient Tracking for Distributed Optimization} 

\author{Eduardo Sebastián, Mauro Franceschelli, Andrea Gasparri, Eduardo Montijano and Carlos Sagüés
\thanks{This work was supported by the Spanish projects PID2021-125514NB-I00, PID2021-124137OB-I00, TED2021-130224B-I00
funded by MCIN/AEI/10.13039/501100011033, by ERDF A way of making Europe and by the
European Union NextGenerationEU/PRTR, by the Gobierno de Aragón under Project DGA T45-23R, and by Spanish grant FPU19-05700.}%
\thanks{E. Sebastián, E. Montijano and C. Sagüés are with the Instituto de Investigaci\'on en Ingenier\'ia de Arag\'on, Universidad de Zaragoza, Spain 
(email:\texttt{\scriptsize esebastian@unizar.es, emonti@unizar.es, csagues@unizar.es}). M. Franceschelli is with the Department of Electrical and Electronic
Engineering, University of Cagliari, Italy (email: \texttt{\scriptsize mauro.franceschelli@unica.it}). A. Gasparri is with the Department of Engineering, Roma Tre University, Italy (email:\texttt{\scriptsize andrea.gasparri@uniroma3.it}).}
}

\newcommand\copyrighttext{%
  \footnotesize \textcopyright This paper has been accepted for publication at IEEE Control Systems Letters. Please, when citing the paper, refer to the official manuscript published by the IEEE Control Systems Letters.}
\newcommand\copyrightnotice{%
\begin{tikzpicture}[remember picture,overlay]
\node[anchor=south,yshift=10pt] at (current page.south) {\fbox{\parbox{\dimexpr\textwidth-\fboxsep-\fboxrule\relax}{\copyrighttext}}};
\end{tikzpicture}%
}

\begin{document}

\maketitle
\copyrightnotice
\thispagestyle{empty}

\begin{abstract}
This paper presents a novel accelerated distributed algorithm for \blue{unconstrained} consensus optimization \blue{over static undirected networks}. The proposed algorithm combines the benefits of acceleration from momentum, the robustness of the alternating direction method of multipliers, and the computational efficiency of gradient tracking to surpass existing state-of-the-art methods in convergence speed, while preserving their computational and communication cost. First, we prove that, by applying momentum on the average dynamic consensus protocol over the estimates and gradient, we can study the algorithm as an interconnection of two singularly perturbed systems: the outer system connects the consensus variables and the optimization variables, and the inner system connects the estimates of the optimum and the auxiliary optimization variables. Next, we prove that, by adding momentum to the auxiliary dynamics, our algorithm always achieves faster convergence than the achievable linear convergence rate for the non-accelerated alternating direction method of multipliers gradient tracking algorithm case. Through simulations, we numerically show that our accelerated algorithm
\blue{surpasses the existing accelerated and non-accelerated 
distributed consensus first-order optimization protocols in convergence speed}.
\end{abstract}



\section{Introduction}\label{sec:intro}
\IEEEPARstart{D}{istributed} optimization refers to the problem of finding the global optimum of an optimization problem where the cost function, the constraints or the available information is distributed over a network \cite{nedic2018network}. Of particular interest is the case of unconstrained consensus optimization \cite{nedic2018distributed}, where the global cost function is the sum of local cost functions. Practical instances of such a setting are federated learning \cite{nedic2020distributed}, networked games \cite{parise2019variational} or multi-robot control \cite{shorinwa2024distributed}. In all these applications, it is of key importance to reconstruct the global optimum as fast as possible to reduce the number of gradient computations, the communication efforts and adapt to evolving environments. 
To cope with this, we propose a novel accelerated distributed optimization algorithm that exploits momentum to \blue{outperform existing first order approaches in convergence speed}, while leveraging the robustness of the alternating direction method of multipliers (ADMM) and the computational simplicity of gradient tracking (GT).

The evolution of distributed optimization algorithms can be divided in three stages. The first approaches departed from subgradient \cite{sundhar2010distributed}, gradient \cite{jakovetic2014fast} and proximal methods \cite{bertsekas2011incremental} to derive distributed versions that achieve convergence to the optimum if and only if the step size diminishes with time. This property poses a fundamental trade-off between speed and accuracy because fixed step sizes lead to convergence within a given distance from the optimum. 

To overcome the speed-accuracy trade-off, several works capitalized from different points of view the properties of dynamic average consensus \cite{franceschelli2019multi, sebastian2023accelerated} and splitting methods to develop GT \blue{\cite{shi2015extra,qu2017harnessing,nedic2017achieving,xin2019distributed}} and ADMM or primal/dual methods \cite{wei20131,bastianello2020asynchronous,8472154} for distributed optimization. In both cases, linear convergence is achieved for constant step sizes; however, these approaches suffer from their respective caveats. GT is simple to compute but is inherently non-robust against non-ideal affections such as initialization errors, gradient noise or communication disturbances \blue{\cite{bin2022stability}}. ADMM is robust against those undesirable effects but, as a proximal method, it has a non-negligible computational cost at each iteration \blue{\cite{bastianello2020asynchronous}}. In both cases, due to the importance of convergence speed, different works have exploited acceleration methods to improve convergence rate such as momentum \blue{\cite{nguyen2023accelerated,huang2024accelerated}}, Nesterov acceleration \cite{qu2019accelerated}, \blue{Anderson acceleration and adaptive preconditioning for proximal-based optimization such as ADMM and Douglas-Rachford Splitting \cite{fu2020anderson, adil2021rapid}} or smooth/non-smooth regularizers \cite{li2019decentralized}. \blue{Many of the aforementioned methods work for time-varying and/or directed graphs, whereas in this paper we focus on the static undirected case.}

Recently, a distributed optimization protocol that combines ADMM and GT \cite{carnevale2023distributed,carnevale2023admm} was proposed as an alternative to combine the benefits of both approaches. The key idea is that GT can be reformulated in ADMM form, where the ADMM cost function is quadratic in the optimization variables. Therefore, it is possible to obtain a closed analytical solution that is cheap to compute and robust. Nevertheless, despite proving linear convergence for strongly convex cost functions, the algorithm needs time-scale decoupling between the optimum estimate and auxiliary dynamics, limiting the achievable convergence speed. 

To overcome such limitation, \textbf{our main contribution} (Section \ref{sec:proposal}) is a novel accelerated distributed optimization method \blue{for static undirected networks} based on ADMM and GT which exploits momentum  \blue{(for details about the concept of momentum, see \cite{ghosh1996first})}. Momentum allows to accelerate the convergence of tracking average consensus variables \blue{(e.g., \cite{sebastian2023accelerated})} over the estimates and gradients, and speed up the convergence speed of the auxiliary variables' dynamics. As a consequence, we can 
study the convergence properties of our protocol and show that it always achieves linear convergence to the global optimum. More importantly, we also prove that, by adding momentum, acceleration is guaranteed compared to the same algorithm without momentum. We empirically evaluate our proposed algorithm (Section \ref{sec:experiments}) against existing state-of-the-art algorithms, \blue{showing that ours surpasses all of them in convergence speed}, using the same computation-communication resources \blue{while increasing the memory burden}.

\section{Preliminaries}\label{sec:problem_formulation}
\subsection{Problem Formulation}\label{subsec:problem_formulation}

A network of $\mathsf{N} > 0$ nodes\footnote{\blue{\textbf{Notation.} Let $\mathbb{N}$ and $\mathbb{R}$ be the natural and real number sets. We use $\mathsf{car}(\bullet)$ to denote the cardinality of a set. We define as $\mathbf{I}_p$, $\mathbf{0}_{p\times k}$, $\mathbf{1}_{p\times k}$ the identity matrix of dimension $p$, the $p \times k$ matrix of zeros and the $p \times k$ matrix of ones respectively. Besides, $\mathbf{1}_{N, p} = \mathbf{1}_{N} \otimes \mathbf{I}_p$ where $\otimes$ denotes the Kronecker product. Let $\sigma_i(\bullet)$ be the $i$-th eigenvalue of a matrix and $\det(\bullet)$ its determinant. Let $||\bullet||$ be the L2-norm. We denote by $\mathsf{diag}(\{\mathbf{E}_i\}_{i=1}^N)$ the block diagonal matrix whose $i-$th block element matrix is $\mathbf{E}_i$.}} cooperates to find the global optimum $\mathbf{x}^* \in \mathbb{R}^n$ of the consensus optimization problem
\begin{equation}\label{eq:original}
    \mathbf{x}^* = \arg \min_{\mathbf{x} \in \mathbb{R}^n} \sum_{i=1}^{\mathsf{N}} f_i(\mathbf{x}),
\end{equation}
where $\mathbf{x} \in \mathbb{R}^n$ is the $n-$dimensional decision variable. Each node $i \in \mathcal{V} = \{1, \dots, \mathsf{N}\}$ only has access to its local objective function $f_i(\bullet)$ and the local estimate of global optimum $\mathbf{x}_i^t$, with $t \in \mathbb{N}$ denoting the discrete time steps. Problem \eqref{eq:original} is reformulated in distributed form by including the consensus constraint as follows
\begin{equation}\label{eq:original_2}
    \mathbf{x}^* = \arg \min_{\mathbf{x}_i \in \mathbb{R}^n} \sum_{i=1}^{\mathsf{N}} f_i(\mathbf{x}_i) \quad \text{ s.t. } \mathbf{x}_i = \mathbf{x}_j \quad \forall i,j \in \mathcal{V}.
\end{equation}
We assume that $f_i(\bullet)$ is \textbf{c}-strongly convex and the gradients $\nabla f_i(\mathbf{x})$ are L-Lipschitz continuous $\forall i \in \mathcal{V}$. This assumption guarantees that $\mathbf{x}^*$ is unique. 

The network of nodes is described by an undirected connected graph $\mathcal{G} = \{\mathcal{V}, \mathcal{E}\}$, where  $\mathcal{E} \subseteq \mathcal{V} \times \mathcal{V}$ is the set of edges of the graph, such that $(i, j) \in \mathcal{E}$ means that nodes $i$ and $j$ can communicate with each other. The set of neighbors of node $i$ is $\mathcal{N}_i = \{j | (i,j) \in \mathcal{E}\}$. Note that $i \notin \mathcal{N}_i$. The degree of node $i$ is $d_i = \mathsf{car}(\mathcal{N}_i)$ and $d = \sum_{i=1}^{\mathsf{N}}d_i$.

The goal of this work is to develop a distributed algorithm such that each local estimate $\mathbf{x}_i^t$ linearly converges to the global optimum $\mathbf{x}^*$, respecting the topology structure of $\mathcal{G}$.

\subsection{Singularly Perturbed Systems}\label{subsec:sps}

Some of the mathematical derivations exposed in this work are based on  singularly perturbed systems theory \cite{carnevale2022tracking}. Singularly perturbed systems are dynamical systems composed by two interconnected dynamics, each of them evolving at different time-scales. Typically, the slow dynamics corresponds to the desired states to be controlled (e.g., the estimates $\mathbf{x}_i^t$) while the fast dynamics corresponds to auxiliary variables. To ensure stability, the speed of the slow system is constrained by the speed of the fast dynamics. For clarity, we reproduce a theorem that establishes that, for a sufficiently slow dynamics, the slow system \eqref{eq:slow} is decoupled from the fast system \eqref{eq:fast} and they are both globally exponentially stable. We will use this fact to prove convergence properties of our algorithm.
\begin{theorem}[Theorem II.3, \cite{carnevale2023admm}]\label{th:sps}
    Let
    \begin{subequations}
    \begin{alignat}{2}
        \blue{\blue{\mathbf{p}}^{t+1}} =& \blue{\blue{\mathbf{p}}^t} + \gamma f(\blue{\mathbf{p}}^t, \blue{\mathbf{q}}^t, t)\label{eq:slow}
        \\
        \blue{\mathbf{q}^{t+1}} =& g(\blue{\mathbf{q}}^t, \blue{\mathbf{p}}^t, t),\label{eq:fast}
    \end{alignat}
    \end{subequations}
    with $\blue{\blue{\mathbf{p}}} \in \mathbb{R}^n$, $\blue{\mathbf{q}} \in \mathbb{R}^m$,  $\gamma > 0$ and $f(\bullet), g(\bullet)$ Lipschitz continuous uniformly over $t$. Let assume that there exists a function $\blue{\mathbf{q}_{\mathsf{eq}}}(\mathbf{p})$ that is Lipschitz continuous such that
\begin{subequations}\label{eq:eq_cond}
    \begin{alignat}{2}
        \mathbf{0}_n =& \gamma f(\mathbf{0}_n,\blue{\mathbf{q}_{\mathsf{eq}}}(\mathbf{0}_n), t)
        \\
        \blue{\mathbf{q}_{\mathsf{eq}}}(\mathbf{p}) =& g(\blue{\mathbf{q}_{\mathsf{eq}}}(\mathbf{p}), \blue{\mathbf{p}}, t).
    \end{alignat}
    \end{subequations}
    We call reduced system the system given by
    \begin{equation}
        \blue{\blue{\mathbf{p}}^{t+1}} = \blue{\blue{\mathbf{p}}^t} + \gamma f(\blue{\blue{\mathbf{p}}^t}, \blue{\mathbf{q}_{\mathsf{eq}}}(\mathbf{p}^t), t)
    \end{equation}
    and boundary system layer with $\blue{\tilde{\mathbf{q}}} \in \mathbb{R}^m$ the system 
    \begin{equation}
        \blue{\blue{\tilde{\mathbf{q}}}^{t+1}} = g(\blue{\blue{\tilde{\mathbf{q}}}^{t}} + \blue{\mathbf{q}_{\mathsf{eq}}}(\mathbf{p}^t), \blue{\mathbf{p}}^t, t) - \blue{\mathbf{q}_{\mathsf{eq}}}(\mathbf{p}^t).
    \end{equation}
    Assume that there exist continuous function $u(\bullet)$
    and gain $\bar{\gamma}_1 > 0$ such that, for any $\gamma \in (0, \bar{\gamma}_1)$, there exist $b_1, b_2, b_3, b_4>0$ \blue{and values $\tilde{\mathbf{q}}, \tilde{\mathbf{q}}_1, \tilde{\mathbf{q}}_2 \in \mathbb{R}^m$} such that
    \begin{subequations}
    \begin{alignat}{2}
        \label{eq:cond1}
        & \hspace{-0.3cm}b_1 ||\blue{\tilde{\mathbf{q}}}||^2 \leq u(\blue{\tilde{\mathbf{q}}}, t) \leq b_2 ||\blue{\tilde{\mathbf{q}}}||^2
        \\
        \label{eq:cond3}
        & \hspace{-0.3cm}u(\blue{\blue{\tilde{\mathbf{q}}}^{t+1}}, t+1) - u(\blue{\blue{\tilde{\mathbf{q}}}^{t}}, t) \leq -b_3 ||\blue{\tilde{\mathbf{q}}}||^2
        \\
        \label{eq:cond5}
        & \hspace{-0.3cm}|u(\blue{\tilde{\mathbf{q}}}_{1}, t) \kern -0.1cm - \kern -0.1cm  u(\blue{\tilde{\mathbf{q}}}_{2}, t)| \kern -0.1cm  \leq \kern -0.1cm  b_4 ||\blue{\tilde{\mathbf{q}}}_{1} \kern -0.1cm  - \kern -0.1cm \blue{\tilde{\mathbf{q}}}_{2}|| (||\blue{\tilde{\mathbf{q}}}_{1}|| \kern -0.1cm + \kern -0.1cm ||\blue{\tilde{\mathbf{q}}}_{2}||).
    \end{alignat}
    \end{subequations}
    Also, assume that there exists a continuous function $w(\bullet)$, and gain $\bar{\gamma}_2 > 0$ such that, for any $\gamma \in (0, \bar{\gamma}_2)$, there exist $c_1, c_2, c_3, c_4 >0$ \blue{and values ${\mathbf{p}}, {\mathbf{p}}_1, {\mathbf{p}}_2 \in \mathbb{R}^n$} such that
    \begin{subequations}
    \begin{alignat}{2}
        \label{eq:cond2}
        & \hspace{-0.4cm}c_1 ||{\blue{\mathbf{p}}}||^2 \leq w({\blue{\mathbf{p}}}, t) \leq c_2 ||{\blue{\mathbf{p}}}||^2
        \\
        \label{eq:cond4}
        & \hspace{-0.4cm}w({\blue{\mathbf{p}}}^{t+1}, t+1) - w({\blue{\mathbf{p}}}^{t}, t) \leq -c_3 ||{\blue{\mathbf{p}}}||^2
        \\
        \label{eq:cond6}
        & \hspace{-0.4cm}|w({\blue{\mathbf{p}}_{1}}, t) \kern -0.1cm- \kern -0.1cmw({\blue{\mathbf{p}}_{2}}, t)| \kern -0.1cm\leq\kern -0.1cm c_4 ||{\blue{\mathbf{p}}_{1}} \kern -0.1cm- \kern -0.1cm{\blue{\mathbf{p}}_{2}}|| (||{\blue{\mathbf{p}}_{1}}||\kern -0.1cm + \kern -0.1cm||{\blue{\mathbf{p}}_{2}}||).
    \end{alignat}
    \end{subequations}
    Then, there exists $\bar{\gamma} \in (0, \min(\bar{\gamma}_1, \bar{\gamma}_2)) $ and gains \mbox{$\kappa_1, \kappa_2 > 0$} such that, for any ${\gamma} \in (0, \bar{\gamma}) $ and $(\blue{\mathbf{p}}^0, \mathbf{q}^0)$, it holds that
    \begin{equation}
        \left|\left|\begin{pmatrix}
            \blue{\blue{\mathbf{p}}^t} - \blue{\mathbf{p}}^*
            \\
            \blue{\mathbf{q}^t} - \blue{\mathbf{q}_{\mathsf{eq}}}(\blue{\blue{\mathbf{p}}^t})
        \end{pmatrix}  \right|\right| \leq \kappa_1 \left|\left|\begin{pmatrix}
            \blue{\mathbf{p}}^0 
            \\
            \blue{\mathbf{q}}^0 - \blue{\mathbf{q}_{\mathsf{eq}}}(\blue{\mathbf{p}}^0)
        \end{pmatrix}  \right|\right| e^{-\kappa_2 t},
    \end{equation}
    where $\blue{\mathbf{p}}^* \in \mathbb{R}^n$ is the desired equilibrium of $\blue{\mathbf{p}}$. 
\end{theorem}
\blue{Conditions \eqref{eq:cond1}, \eqref{eq:cond2}, \eqref{eq:cond3}, \eqref{eq:cond4} come from standard Lyapunov stability arguments, whereas \eqref{eq:cond5} and \eqref{eq:cond6} bound the changes in the Lyapunov functions $w(\mathbf{p},t)$ and $u(\tilde{\mathbf{q}},t)$ to apply Lyapunov exponential stability arguments.}

\subsection{ADMM Gradient Tracking}

As introduced in Section \ref{sec:intro}, a recent paper \cite{carnevale2023admm} proposes a first order distributed optimization algorithm that combines ADMM and GT to obtain a robust linearly convergent algorithm. Algorithm \ref{al:original} presents such scheme.
\begin{algorithm}
\caption{Original \cite{carnevale2023admm} \textsf{ADMM-GT} at node $i$}\label{al:original}
\begin{algorithmic}[1]
\STATE \textbf{State of agent}: $\mathbf{x}^{0}_i \in \mathbb{R}^n$,  ${z}^0_{ij} \in \mathbb{R}^{2n} \quad \forall j \in \mathcal{N}_i$
\STATE \textbf{Parameters}: $\gamma  > 0$, $\rho > 0$, $\alpha \in (0,1)$
\FOR{$t=1,\hdots$}
    \STATE $\mathbf{y}^{t+1}_i = \frac{1}{1 + \rho d_i}(\mathbf{x}_i^t + \begin{pmatrix}
        \mathbf{I}_n & \mathbf{0}_{n\times n}
    \end{pmatrix} \sum_{j\in\mathcal{N}_i}\mathbf{z}_{ij}^t)$
    \STATE $\mathbf{s}^{t+1}_i = \frac{1}{1 + \rho d_i}(\nabla f_i(\mathbf{x}_i^t) + \begin{pmatrix}
        \mathbf{0}_{n\times n} &
        \mathbf{I}_n 
    \end{pmatrix} \sum_{j\in\mathcal{N}_i}\mathbf{z}_{ij}^t)$
    \STATE $\mathbf{z}_{ij}^{t+1} = (1-\alpha)\mathbf{z}_{ij}^t - \alpha (\mathbf{z}_{ji}^t - 2 \rho \begin{pmatrix}
        \mathbf{y}^{t+1}_i \\
        \mathbf{s}^{t+1}_i 
    \end{pmatrix} ) \quad \forall j \in \mathcal{N}_i$ 
    \STATE $\mathbf{x}_i^{t+1} = \mathbf{x}_i^t + \gamma(\mathbf{y}^{t+1}_i - \mathbf{x}_i^{t}) - \gamma \mathbf{s}^{t+1}_i$
\ENDFOR 
\end{algorithmic}
\end{algorithm}
Each node $i$ updates four variables. First, $\mathbf{x}_i^t$ is the variable that estimates the global optimum $\mathbf{x}^*$. The variable $\mathbf{y}^t_i$ reconstructs the average of the estimates over the nodes of the network $\frac{1}{\mathsf{N}}\sum_{i=1}^{\mathsf{N}}\mathbf{x}_i^t$ and aims at eliminating the consensus error. The variable $\mathbf{s}^t_i$ reconstructs the average of the gradients $\frac{1}{\mathsf{N}}\sum_{i=1}^{\mathsf{N}}\nabla f_i(\mathbf{x}_i^t)$ and aims at reducing the error $\mathbf{e}^t_i = ||\mathbf{x}_i^t - \mathbf{x}^*||$. Finally, $\mathbf{z}_{ij}^t \in \mathbb{R}^{2n}$ are auxiliary variables that enforce consensus over $\mathbf{y}^t_i$ and $\mathbf{s}^t_i$ \cite{bastianello2020asynchronous}.  

One fundamental aspect of Algorithm \ref{al:original} is that, by conveniently rewriting the dynamics of \mbox{$\mathbf{x}^t = [(\mathbf{x}^t_1)^{\top}, \hdots, (\mathbf{x}^t_{\mathsf{N}})^{\top}]^{\top}$} and $\mathbf{z}^t = [(\mathbf{z}^t_1)^{\top}, \hdots, (\mathbf{z}^t_{\mathsf{N}})^{\top}]^{\top}$ (with $\mathbf{z}^t_i = [(\mathbf{z}^t_{i1})^{\top}, \hdots, (\mathbf{z}^t_{id_i})^{\top}]^{\top}$), it can be shown that the overall dynamics of the scheme have the structure of a time-varying singularly perturbed system \cite{carnevale2023distributed}. Therefore, relying on Theorem \ref{th:sps}, for a sufficiently small $\gamma$, Algorithm~\ref{al:original} converges linearly to the global optimum of \eqref{eq:original}.

In the following section, we present an algorithm that, by adding momentum to the dynamics of $\mathbf{y}^t_i$, $\mathbf{s}^t_i$ and $\mathbf{z}^t_{ij}$, speeds up the convergence while preserving the computation and communication properties of Algorithm \ref{al:original}.

\section{Accelerated ADMM Gradient Tracking}\label{sec:proposal}
The proposed accelerated distributed optimization scheme is presented in Algorithm \ref{al:accelerated}. 
\begin{algorithm}
\caption{\textsf{A2DMM-GT}  at node $i$}\label{al:accelerated}
\begin{algorithmic}[1]
\STATE \textbf{State of agent}: $\mathbf{x}^{0}_i \in \mathbb{R}^n$,  $\mathbf{z}^0_{ij} \in \mathbb{R}^{2n} \quad \forall j \in \mathcal{N}_i$, \mbox{$\mathbf{z}^{-1}_{ij} \in \mathbb{R}^{2n} \quad \forall j \in \mathcal{N}_i$}, $\mathbf{y}^{0}_i \in \mathbb{R}^n$, $\mathbf{s}^{0}_i \in \mathbb{R}^n$
\STATE \textbf{Parameters}: $\gamma  > 0$, $\rho > 0$, $\alpha \in (0,1)$, $\lambda \in (1, 2)$, $\mu \in (1, 2)$, $\epsilon \in (0, 1)$
\FOR{$t=1,\hdots$}
    \STATE $\bar{\mathbf{y}}^{t+1}_i = \frac{1}{1 + \rho d_i}(\mathbf{x}_i^t + \begin{pmatrix}
        \mathbf{I}_n & \mathbf{0}_{n\times n}
    \end{pmatrix} \sum_{j\in\mathcal{N}_i}\mathbf{z}_{ij}^t)$
    \STATE $\bar{\mathbf{s}}^{t+1}_i = \frac{1}{1 + \rho d_i}(\nabla f_i(\mathbf{x}_i^t) + \begin{pmatrix}
        \mathbf{0}_{n\times n} &
        \mathbf{I}_n 
    \end{pmatrix} \sum_{j\in\mathcal{N}_i}\mathbf{z}_{ij}^t)$
    \STATE $\mathbf{y}^{t+1}_i = \mathbf{y}^{t}_i + \lambda(\bar{\mathbf{y}}^{t+1}_i - \mathbf{y}^{t}_i)$
    \STATE $\mathbf{s}^{t+1}_i = \mathbf{s}^{t}_i + \lambda(\bar{\mathbf{s}}^{t+1}_i - \mathbf{s}^{t}_i)$
    \STATE $\bar{\mathbf{z}}_{ij}^{t+1} = (1-\alpha)\mathbf{z}_{ij}^t - \alpha (\mathbf{z}_{ji}^t - 2 \rho \begin{pmatrix}
        \mathbf{y}^{t+1}_i \\
        \mathbf{s}^{t+1}_i 
    \end{pmatrix} ) \quad \forall j \in \mathcal{N}_i$ 
    \STATE $\mathbf{z}_{ij}^{t+1} = \mu(\mathbf{z}_{ij}^{t} + \epsilon(\bar{\mathbf{z}}_{ij}^{t+1}-\mathbf{z}_{ij}^{t})) + (1 - \mu) \mathbf{z}_{ij}^{t-1} \quad \forall j \in \mathcal{N}_i$
    \vspace{-0.3cm}
    \STATE $\mathbf{x}_i^{t+1} = \mathbf{x}_i^t + \gamma(\mathbf{y}^{t+1}_i - \mathbf{x}_i^{t}) - \gamma \mathbf{s}^{t+1}_i$
\ENDFOR 
\end{algorithmic}
\end{algorithm}
Algorithm \ref{al:accelerated} adds momentum to the dynamics of $\mathbf{y}_i^t$ and $\mathbf{s}_i^t$ (lines 6 and 7) and the dynamics of $\mathbf{z}_{ij}^{t}$ (line 9). The former allows to write Algorithm \ref{al:accelerated} as a chain of two singularly perturbed systems. \blue{The outer singularly perturbed system interconnects the dynamics given by $\mathbf{y}_i^t$ and $\mathbf{s}_i^t$ with the dynamics given by $\mathbf{x}_i^t$ and $\mathbf{z}_i^{t}$. The inner singularly perturbed system interconnects the dynamics given by $\mathbf{z}_i^{t}$ with the dynamics given by $\mathbf{x}_i^t$. In terms of convergence speed, there are three time-scales, from the fastest to the slowest: (1) $\mathbf{y}_i^t, \bar{\mathbf{y}}_i^t, \mathbf{s}_i^t, \bar{\mathbf{s}}_i^t$; (2) $\mathbf{z}_i^t, \bar{\mathbf{z}}_i^t$; and (3) $\mathbf{x}_i^t$.} This time-scale separation leads to a novel convergence analysis of the \textsf{ADMM-GT} algorithm, discovering a limitation in the maximum achievable convergence rate of the dynamics of $\mathbf{z}_i^{t}$. 
By adding momentum to $\mathbf{z}_i^{t}$ this bound is overcome, enhancing the overall convergence rate when compared to Algorithm \ref{al:original}. \blue{Additionally, by respecting the ADMM gradient tracking structure of Algorithm \ref{al:original}, Algorithm \ref{al:accelerated} preserves the robustness from ADMM proved in \cite{carnevale2023admm} and the computational simplicity of gradient tracking methods (the momentum operations are inexpensive to compute).}
Comparing Algorithms \ref{al:original} and \ref{al:accelerated}, their communication cost is the same since only the $\mathbf{z}_{ij}^t$ are exchanged. \blue{In terms of memory, each node has to additionally store $\mathbf{z}_{ij}^{t-1}$, $\mathbf{y}_{i}^{t}$, $\mathbf{s}_{i}^{t}$.}  

To analyze the convergence properties of Algorithm \ref{al:accelerated}, we define the following matrices:
\begin{equation*}
\begin{aligned}
    &\mathbf{A}_{x} \kern -0.1cm = \kern -0.1cm\mathsf{diag}\left(\kern -0.1cm\left\{\kern -0.1cm\begin{pmatrix}
    \mathbf{1}_{d_i,n} 
    \\
    \mathbf{0}_{d_i,n}
    \end{pmatrix}
    \right\}_{i=1}^{\mathsf{N}}\right), \mathbf{A}_{z} \kern -0.1cm = \kern -0.1cm\mathsf{diag}\left(\kern -0.1cm\left\{\kern -0.1cm\begin{pmatrix}
    \mathbf{0}_{d_i,n} 
    \\
    \mathbf{1}_{d_i,n}
    \end{pmatrix}
    \right\}_{i=1}^{\mathsf{N}}\right)
    \\&
    \mathbf{A} = \mathsf{diag}\left(\left\{\mathbf{1}_{d_i, 2n}\right\}_{i=1}^{\mathsf{N}}\right),
    \mathbf{H} = \mathsf{diag}\left(\left\{\frac{\mathbf{I}_n}{1 + \rho d_i}\right\}_{i=1}^{\mathsf{N}}\right). 
\end{aligned}
\end{equation*}
Besides, $\mathbf{P} \in \{0, 1\}^{2nd \times 2nd}$ is a permutation matrix that exchanges \blue{auxiliary variable $\mathbf{z}^t_{ij}$ with auxiliary variable $\mathbf{z}^t_{ji}$ and implements the pairwise message passing found in line 8 of Algorithm \ref{al:accelerated}}, $\mathbf{x}^t = [(\mathbf{x}^t_1)^{\top}, \hdots, (\mathbf{x}^t_{\mathsf{N}})^{\top}]^{\top}$, \mbox{$\mathbf{z}^t = [(\mathbf{z}^t_1)^{\top}, \hdots, (\mathbf{z}^t_{\mathsf{N}})^{\top}]^{\top}$}, $\mathbf{z}^t_i = [(\mathbf{z}^t_{i1})^{\top}, \hdots, (\mathbf{z}^t_{id_i})^{\top}]^{\top}$, \mbox{$\mathbf{y}^t = [(\mathbf{y}^t_1)^{\top}, \hdots, (\mathbf{y}^t_{\mathsf{N}})^{\top}]^{\top}$} and $\mathbf{s}^t = [(\mathbf{s}^t_1)^{\top}, \hdots, (\mathbf{s}^t_{\mathsf{N}})^{\top}]^{\top}.$
Finally, let $\mathbf{g}(\mathbf{x}^t) = [(\nabla f_1(\mathbf{x}^t_1))^{\top},\dots, (\nabla f_{\mathsf{N}}(\mathbf{x}^t_{\mathsf{N}}))^{\top}]^{\top}$ and $\mathbf{v}(\mathbf{y}^t,\mathbf{s}^t) = [(\mathbf{y}^t_1)^{\top},(\mathbf{s}^t_1)^{\top},\dots,(\mathbf{y}^t_{\mathsf{N}})^{\top},(\mathbf{s}^t_{\mathsf{N}})^{\top}]^{\top}$.

Exploiting these definitions, Algorithm \ref{al:accelerated} for the whole network can be written as
\begin{subequations}\label{eq:proposed_global}
    \begin{alignat}{2}
        \hspace{-0.2cm}\mathbf{y}^{t+1} =& (1 - \lambda) \mathbf{y}^{t} + \lambda \mathbf{H}(\mathbf{x}^t + \mathbf{A}_x^{\blue{\top}} \mathbf{z}^t),
        \\
        \hspace{-0.2cm}\mathbf{s}^{t+1} =& (1 - \lambda) \mathbf{s}^{t} + \lambda \mathbf{H}(\mathbf{g}(\mathbf{x}^t) + \mathbf{A}_z^{\blue{\top}} \mathbf{z}^t),
        \\
        \hspace{-0.2cm}\mathbf{z}^{t+1} =& \mu \mathbf{F} \mathbf{z}^{t} \kern -0.1cm + \kern -0.1cm (1-\mu)\mathbf{z}^{t-1} \kern -0.1cm + \kern -0.1cm2\mu\epsilon\alpha\rho\mathbf{P}\mathbf{A}\mathbf{v}(\mathbf{y}^{t+1},\mathbf{s}^{t+1}),\label{eq:zzz}
        \\ 
        \hspace{-0.2cm}\mathbf{x}^{t+1} =& \mathbf{x}^{t} + \gamma(\mathbf{y}^{t+1} - \mathbf{x}^{t}) - \gamma \mathbf{s}^{t+1}
    \end{alignat}
\end{subequations}
where $\mathbf{F} = (1 - \epsilon\alpha)\mathbf{I}_{2nd} - \epsilon\alpha\mathbf{P}$. The compact expression in \eqref{eq:proposed_global} can be reformulated as an interconnection of two singularly perturbed systems. In Fig. \ref{fig:interconnection} we represent the interconnections between the dynamics of $\mathbf{x}^t$, $\mathbf{z}^t$, $\mathbf{y}^t$ and $\mathbf{s}^t$. The inner system (green) 
connects the dynamics of $\mathbf{x}^t$ (slow) and the dynamics of $\mathbf{z}^t$ (fast). Indeed, if $\mu=\epsilon=1$, the inner system corresponds to the singularly perturbed system studied in \cite{carnevale2023admm}. On the other hand, the outer system is also a singularly perturbed system given by $\mathbf{r}^t = [(\mathbf{y}^t)^{\top}, (\mathbf{s}^t)^{\top}]^{\top}$, $\mathbf{w}^t = [(\mathbf{x}^t)^{\top}, (\mathbf{z}^t)^{\top}, (\mathbf{z}^{t-1})^{\top}]^{\top}$,
$\mathbf{v}(\mathbf{r}^t)= \mathbf{v}(\mathbf{y}^t,\mathbf{s}^t)$, and
\begin{equation}\label{eq:fast_expanded}
        \begin{aligned} 
        \mathbf{r}^{t+1} =& (1 - \lambda) \mathbf{r}^{t} + \lambda \begin{pmatrix}
            \mathbf{H}(\mathbf{x}^t + \mathbf{A}_x \mathbf{z}^t)
            \\
            \mathbf{H}(\mathbf{g}(\mathbf{x}^t) + \mathbf{A}_z \mathbf{z}^t)
        \end{pmatrix} = 
        \\&
        (1 - \lambda) \mathbf{r}^{t} + \lambda \mathbf{h}(\mathbf{w}^t),
        \end{aligned}
\end{equation}
\begin{equation}\label{eq:slow_expanded}
        \begin{aligned}
        \mathbf{w}^{t+1} =& 
        \begin{pmatrix}
            (1-\gamma)\mathbf{I}_{\blue{\mathsf{N}n}} & \mathbf{0}_{\blue{\mathsf{N}n}\times2nd} & \mathbf{0}_{\blue{\mathsf{N}n}\times2nd}
            \\
            \mathbf{0}_{2nd\times2nd} &
            \mu \mathbf{F} & (1-\mu)\mathbf{I}_{2nd}
            \\
            \mathbf{0}_{2nd\times2nd}
            &
            \mathbf{I}_{2nd}
            &
            \mathbf{0}_{2nd\times2nd}
        \end{pmatrix}
        \mathbf{w}^{t} + 
        \\&
        \begin{pmatrix}
            \gamma\begin{pmatrix}
                \mathbf{I}_{\blue{\mathsf{N}n}} & -\mathbf{I}_{\blue{\mathsf{N}n}}
            \end{pmatrix} \mathbf{r}^t
            \\
            2\mu\epsilon\alpha\rho\mathbf{P}\mathbf{A}\mathbf{v}(\mathbf{r}^t)
            \\
            \mathbf{0}_{2nd}
        \end{pmatrix} = 
        \mathbf{Q}\mathbf{w}^t + \mathbf{B}(\mathbf{r}^t).
        \end{aligned}
\end{equation}
Given an arbitrary constant $\mathbf{w}^t = \mathbf{w}$ for all $t$, let function  $\mathbf{r}_{\mathsf{eq}}(\mathbf{w})$ from \eqref{eq:fast_expanded} be
    \begin{equation}\label{eq:eq}
        \mathbf{r}_{\mathsf{eq}}(\mathbf{w}) = (1-\lambda)\mathbf{r}_{\mathsf{eq}}(\mathbf{w}) + \lambda \blue{\mathbf{h}(\mathbf{w})} \Rightarrow \mathbf{r}_{\mathsf{eq}}(\mathbf{w}) = \blue{\mathbf{h}(\mathbf{w})}.
\end{equation} 
Then, we define the error coordinates $\tilde{\mathbf{r}}^{t} = \mathbf{r}^t - \mathbf{r}_{\mathsf{eq}}(\mathbf{w})$. Regarding the slow system dynamics \eqref{eq:slow_expanded}, we define the desired equilibrium $\mathbf{w}^*$ as
\begin{equation}\label{eq:desired_def}
        {\mathbf{w}}^* = 
        \begin{pmatrix}
            {\mathbf{x}}^* 
            \\
            2(\mathbf{I}_{2nd} + \mathbf{P})^{-1}\rho \mathbf{P}\mathbf{A}\mathbf{v}(\mathbf{r}_{\mathsf{eq}}(\mathbf{w}^t))
            \\ 
            2(\mathbf{I}_{2nd} + \mathbf{P})^{-1}\rho \mathbf{P}\mathbf{A}\mathbf{v}(\mathbf{r}_{\mathsf{eq}}(\mathbf{w}^t))
        \end{pmatrix},
    \end{equation}
where we have exploited the fact that $\mathbf{z}^*$ can be arbitrarily chosen, making it coincide with the value of $\mathbf{z}^t$ when \mbox{$\mathbf{z}^t = \mathbf{z}^{t-1}$}, \blue{which is the steady-state value of $\mathbf{z}^t$ in Eq. \eqref{eq:zzz}}. Finally, we define the error dynamics \mbox{$\tilde{\mathbf{w}}^t = {\mathbf{w}}^t - {\mathbf{w}}^*$}.

We aim at exploiting Theorem \ref{th:sps} to prove the convergence properties of Algorithm \ref{al:accelerated}. To do so, we first study the convergence properties of the boundary system layer error dynamics $\tilde{\mathbf{r}}^{t+1}$ and the reduced system error dynamics $\tilde{\mathbf{w}}^{t+1}$ by drawing their equivalence with the fast \eqref{eq:fast} and slow \eqref{eq:slow} dynamics respectively, proving that conditions \eqref{eq:cond1}-\eqref{eq:cond5} and \eqref{eq:cond2}-\eqref{eq:cond6} hold. In Lemmas \ref{prop:u} and \ref{prop:w} we prove that $\tilde{\mathbf{r}}^{t+1}$ can be written in the form of Eq. \eqref{eq:fast} and $\tilde{\mathbf{w}}^{t+1}$ in the form of Eq. \eqref{eq:slow}. After that, Proposition \ref{prop:ges} resorts to Theorem \ref{th:sps} to prove the global exponential stability of the dynamics of $\tilde{\mathbf{r}}^{t+1}$ and $\tilde{\mathbf{w}}^{t+1}$ and, consequently, the global exponential stability of Algorithm \ref{al:accelerated}. Finally, we use Lemma \ref{prop:w} to demonstrate that Algorithm \ref{al:accelerated} leads to acceleration with respect to Algorithm \ref{al:original}, formalized through Theorem \ref{prop:rate}, where we substitute the slow state from $\mathbf{p}$ to $\tilde{\mathbf{w}}$ and the fast state from $\mathbf{q}$ to $\tilde{\mathbf{r}}$. 

\begin{figure}
    \centering
    \includegraphics[width=1\columnwidth]{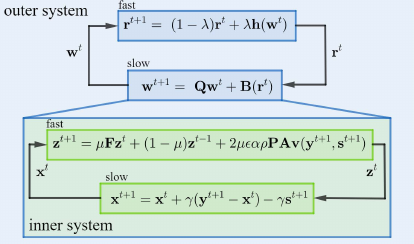}
    \caption{Representation of Algorithm \ref{al:accelerated} as an interconnection of two singularly perturbed systems. The outer singularly perturbed system connects the dynamics of the consensus variables $\mathbf{r}^t$ and the estimates/auxiliary variables $\mathbf{w}^t$. The inner singularly perturbed system connects the dynamics of the estimates $\mathbf{x}^t$ and the auxiliary variables $\mathbf{z}^t$.}
    \label{fig:interconnection}
\end{figure}

\begin{lemma}\label{prop:u}
    Let $\mathbf{w}^t = \mathbf{w}$. There exists a continuous function $u(\tilde{\mathbf{r}}^{t},t)$ such that, for $\lambda \in (1,2)$, \eqref{eq:cond1}, \eqref{eq:cond3} and \eqref{eq:cond5} hold.
\end{lemma}
\begin{proof}
    The boundary system layer dynamics of \eqref{eq:fast_expanded} is
    \begin{equation}
        \tilde{\mathbf{r}}^{t+1} \kern -0.1cm = \kern -0.1cm \mathbf{r}^{t+1} - \mathbf{r}_{\mathsf{eq}}(\mathbf{w}) \kern -0.1cm = \kern -0.1cm (1 - \lambda)(\tilde{\mathbf{r}}^{t} + \mathbf{r}_{\mathsf{eq}}(\mathbf{w})) + \lambda \mathbf{h}(\mathbf{w}^t) - \mathbf{r}_{\mathsf{eq}}(\mathbf{w}),
    \end{equation}
    from which, using \eqref{eq:eq}, it is derived that $\tilde{\mathbf{r}}^{t+1} = (1 - \lambda)\tilde{\mathbf{r}}^{t}$.
    Thus, if $\lambda \in (1, 2)$, then the error dynamics of $\tilde{\mathbf{r}}^{t+1}$ is globally exponentially stable and the conditions \eqref{eq:cond1}, \eqref{eq:cond3} and \eqref{eq:cond5} are satisfied by choosing as Lyapunov function $u(\tilde{\mathbf{r}}^{t},t) = \frac{1}{2}(\tilde{\mathbf{r}}^{t})^{\top}\tilde{\mathbf{r}}^{t}$.
\end{proof}
\begin{lemma}\label{prop:w}
    Let $\mathbf{r}^t = \mathbf{r}_{\mathsf{eq}}(\mathbf{w}^t)$, $\lambda, \mu \in (1, 2)$, $\gamma > 0$, \mbox{$\epsilon, \alpha \in (0, 1)$} and $\rho>0$. There exists a continuous function $w(\tilde{\mathbf{w}}^t, t)$ such that \eqref{eq:cond2}, \eqref{eq:cond4} and \eqref{eq:cond6} hold.
\end{lemma}
\begin{proof}
    First, the error dynamics $\tilde{\mathbf{w}}^{t+1}$ are 
    \begin{equation}
        \begin{aligned}
        \tilde{\mathbf{w}}^{t+1} =& 
        \mathbf{Q}
        (\tilde{\mathbf{w}}^{t} + \mathbf{w}^*) + 
        \mathbf{B}(\mathbf{r}_{\mathsf{eq}}(\mathbf{w}^t))  - \mathbf{w}^*
        \end{aligned} \Rightarrow
    \end{equation}
    \begin{equation} \label{eq:slow_error}
        \tilde{\mathbf{w}}^{t+1}  \kern -0.1cm=\kern -0.1cm 
        \mathbf{Q} \tilde{\mathbf{w}}^t \kern -0.05cm+ \kern -0.05cm
        \begin{pmatrix} 
            -\gamma\mathbf{x}^*
            \\
            \mu\mathbf{F}\mathbf{z}^* -\mu\mathbf{z}^*
            \\
            \mathbf{z}^* - \mathbf{z}^*
        \end{pmatrix}\kern -0.05cm
         +  \kern -0.05cm
        \mathbf{B}(\mathbf{r}_{\mathsf{eq}}(\mathbf{w}^t)).
    \end{equation}
    By substituting the definition of $\mathbf{z}^*$ from Eq. \eqref{eq:desired_def} \blue{in matrix $\mathbf{B}(\mathbf{r}_{\mathsf{eq}}(\mathbf{w}^t))$ of} Eq. \eqref{eq:slow_error}, we obtain that
    \begin{equation} \label{eq:w_error_dynamics}
        \tilde{\mathbf{w}}^{t+1} =
        \mathbf{Q} \tilde{\mathbf{w}}^t + 
        \begin{pmatrix}
            -\gamma\mathbf{x}^* + \gamma\begin{pmatrix}
                \mathbf{I}_{\blue{\mathsf{N}n}} & -\mathbf{I}_{\blue{\mathsf{N}n}}
            \end{pmatrix} \mathbf{h}(\mathbf{w}^t)
            \\
            \mathbf{0}_{2nd}
            \\
            \mathbf{0}_{2nd}
        \end{pmatrix}.
    \end{equation}
    According to \eqref{eq:w_error_dynamics}, the error dynamics associated to $\mathbf{z}^{t+1}$ and $\mathbf{z}^{t}$ do not depend on $\mathbf{x}^t$
    \begin{equation}\label{eq:z_error_dynamics}
        \kern -0.1cm
        \begin{pmatrix}
            \tilde{\mathbf{z}}^{t+1}
            \\
            \tilde{\mathbf{z}}^{t}
        \end{pmatrix} 
        \kern -0.1cm
        =
        \kern -0.1cm
        \begin{pmatrix}
            \mu \mathbf{F} & (1-\mu)\mathbf{I}_{2nd}
            \\
            \mathbf{I}_{2nd} & \mathbf{0}_{2nd\times 2nd}
        \end{pmatrix} 
        \kern -0.1cm
        \begin{pmatrix}
            \tilde{\mathbf{z}}^{t}
            \\
            \tilde{\mathbf{z}}^{t-1}
        \end{pmatrix} 
        \kern -0.1cm
        = 
        \kern -0.1cm
        \mathbf{L} 
        \begin{pmatrix}
            \tilde{\mathbf{z}}^{t}
            \\
            \tilde{\mathbf{z}}^{t-1}
        \end{pmatrix}, 
    \end{equation}
    so they can be independently analyzed. \blue{Moreover, \eqref{eq:z_error_dynamics} can be rewritten in the form of \eqref{eq:slow} by adding and subtracting the identity matrix from $\mathbf{L}$.} On the other hand, the error dynamics associated to $\mathbf{x}^t$
    \begin{equation}\label{eq:x_error_dynamics}
        \kern -0.05cm\tilde{\mathbf{x}}^{t+1} \kern -0.05cm=\kern -0.05cm \tilde{\mathbf{x}}^{t} \kern -0.05cm+\kern -0.05cm \gamma (\mathbf{H}(\mathbf{x}^t \kern -0.05cm+\kern -0.05cm \mathbf{A}_x\mathbf{z}^t)\kern -0.05cm-\kern -0.05cm\mathbf{x}^t) \kern -0.05cm-\kern -0.05cm \gamma \mathbf{H}(\mathbf{g}(\mathbf{x}^t) \kern -0.05cm+\kern -0.05cm \mathbf{A}_z \mathbf{z}^t)
    \end{equation}
    are the error dynamics of the reduced system of the inner singularly perturbed system, which are the same to those analyzed in \cite{carnevale2023admm}. \blue{In particular, Theorem III.1 in \cite{carnevale2023admm} proves that Algorithm \ref{al:original} is globally exponentially stable when $\rho>0$, $\alpha \in (0,1)$ and $\gamma>0$ sufficiently small. Conveniently, the proof of Theorem III.1 in \cite{carnevale2023admm} refers to the same dynamics in \eqref{eq:x_error_dynamics}, and therefore the result can be applied.} However, this result holds only if the dynamics in \eqref{eq:z_error_dynamics} are also globally exponentially stable. Otherwise, Theorem \ref{th:sps} applied to the dynamics of $\mathbf{x}^t$ and $\mathbf{z}^t$ does not hold. Therefore, the next step is to prove the global exponential stability of \eqref{eq:z_error_dynamics}. The error dynamics in \eqref{eq:z_error_dynamics} are globally exponentially stable if and only if $\{|\sigma_l(\mathbf{L})| < 1\}_{l=1}^{4nd}$. Taking into account that $\det(\mathbf{L} - \sigma\mathbf{I}_{4nd}) = 0$ and that $\mathbf{F} = (1 - \epsilon\alpha)\mathbf{I}_{2nd} - \epsilon\alpha\mathbf{P}$, the eigenvalues of $\mathbf{L}$ are
        \begin{equation}\label{eq:eigen_expression}
            -\sigma^2 + \sigma(\mu - \mu\alpha\epsilon \pm \mu\alpha\epsilon) + (1 - \mu) = 0.
        \end{equation}
        The $\pm$ in \eqref{eq:eigen_expression} comes from the fact that $\det(\mathbf{P}) = (-1)^p$, where $p$ is the number of permutations. If $p$ is even,  \begin{equation}\label{eq:eigen_expression_2}
            -\sigma^2 + \mu\sigma + (1 - \mu) = 0 \Rightarrow \sigma = \frac{\mu}{2} \pm \sqrt{\frac{\mu^2}{4} + (1 - \mu)}.
        \end{equation}
        From \eqref{eq:eigen_expression_2}, $-1 < \sigma < 1$. If $\mu \in (1,2)$. If $p$ is odd, then \begin{equation}\label{eq:eigen_expression_3}
            -\sigma^2 + \mu(1 - 2\alpha\epsilon)\sigma + (1 - \mu) = 0.
        \end{equation} 
        Equation \eqref{eq:eigen_expression_3} implies that
        \begin{equation}\label{eq:eigen_expression_4}
            -1 < \frac{\mu(1 - 2\epsilon\alpha)}{2} \pm \sqrt{\frac{\mu^2}{4}(1-2\epsilon\alpha)^2 + (1 - \mu)} < 1.
        \end{equation} 
        By solving the four inequality equations in \eqref{eq:eigen_expression_4} and taking the more restrictive conditions, we obtain that $\epsilon\alpha \in (0, 1)$, and since $\alpha \in (0, 1)$, then $\epsilon \in (0, 1)$. Therefore, if \mbox{$\epsilon \in (0, 1)$} and $\mu \in (1, 2)$, then the error dynamics in \eqref{eq:z_error_dynamics} are globally exponentially stable. This implies that the error dynamics in \eqref{eq:x_error_dynamics} are globally exponentially stable as well and the conditions for the inner singularly perturbed system in Theorem \ref{th:sps} hold. Subsequently, the error dynamics in \eqref{eq:w_error_dynamics} are globally exponentially stable and there exist a Lyapunov function $w(\tilde{\mathbf{w}}^t, t)$ such that \eqref{eq:cond2}, \eqref{eq:cond4} and \eqref{eq:cond6} are satisfied.  
\end{proof}
\begin{proposition}\label{prop:ges}
    Consider Algorithm \ref{al:accelerated} with $\rho>0$, \mbox{$\alpha, \epsilon \in (0, 1)$} and $\lambda, \mu \in (1, 2)$. Then, there exists \mbox{$\bar{\gamma}, c_1, c_2>0$} such that, for any $\gamma \in (0, \bar{\gamma})$, $(\mathbf{x}^0, \mathbf{z}^0, \mathbf{z}^{-1}, \mathbf{y}^0, \mathbf{s}^0)$, it holds that
    $
    ||\mathbf{x}^t - \mathbf{1}^{\top}_{\mathsf{N},n}\mathbf{x}^*|| \leq c_1 e^{-c_2 t}.
    $
\end{proposition}
\begin{proof}
Let $\tilde{\mathbf{w}}^t$ correspond to the error state of the slow dynamics in \eqref{eq:slow} and $\tilde{\mathbf{r}}^t$ the error state of the fast error dynamics in \eqref{eq:fast}. According to Lemmas \ref{prop:u} and \ref{prop:w}, the conditions \eqref{eq:cond1}-\eqref{eq:cond6} hold for the error dynamics of $\tilde{\mathbf{w}}^t$ and $\tilde{\mathbf{r}}^t$  respectively. Besides, from the definition of \blue{$\mathbf{r}_{\mathsf{eq}}(\mathbf{w}^t)$} it can be verified that condition \eqref{eq:eq_cond} holds. Then, Theorem \ref{th:sps} can be invoked to prove the global exponential stability of $\tilde{\mathbf{w}}^t$ and $\tilde{\mathbf{r}}^t$. By changing the coordinates from $\tilde{\mathbf{w}}^t$ and $\tilde{\mathbf{r}}^t$ to ${\mathbf{w}}^t$ and ${\mathbf{r}}^t$, the proof ends.  
\end{proof}

Proposition \ref{prop:ges} establishes that Algorithm \ref{al:accelerated} linearly converges to the global optimum $\mathbf{x}^*$. Nevertheless, acceleration is not proved yet. This is proved in the next theorem. 

\begin{theorem}\label{prop:rate}
    Consider $\rho >0$, $\alpha \in (0,1)$ and $\lambda \in (1,2)$, \blue{and $\gamma > 0$ sufficiently small}. Also, let $\beta_1$ and $\beta_2$ be the convergence rate of the dynamics of $\mathbf{z}^t$ for Algorithm \ref{al:original} and \ref{al:accelerated} respectively. Then,  Algorithm \ref{al:accelerated} always converges faster than Algorithm \ref{al:original} ($\beta_2 < \beta_1$) for any $\epsilon \in (0,1)$ and $\mu \in (1,2)$.
\end{theorem}
\begin{proof}
    In Lemma \ref{prop:w} Eq. \eqref{eq:z_error_dynamics} it is proven that the dynamics of $\tilde{\mathbf{z}}^t$ are \blue{determined by $\mathbf{L}$}. Therefore, the convergence rate of $\tilde{\mathbf{z}}^t$ and subsequently ${\mathbf{z}}^t$ is given by $\max(\{\blue{|}\sigma_l(\mathbf{L})\blue{|}\}_{l=1}^{4nd})$. Therefore, we have that 
    $$
    \begin{aligned}
        \max(\{\blue{|}\sigma_l(\mathbf{L})\blue{|}&\}_{l=1}^{4nd}) = \max \left(\blue{\Bigg|}\frac{\mu}{2} + \sqrt{\frac{\mu^2}{4} + (1 - \mu)}\blue{\Bigg|}\right., 
        \\
        &\left.\blue{\Bigg|}\frac{\mu(1 - 2\epsilon\alpha)}{2} + \sqrt{\frac{\mu^2}{4}(1-2\epsilon\alpha)^2 + (1 - \mu)}\blue{\Bigg|}\right).
    \end{aligned}
    $$
    Moreover, notice that Algorithm \ref{al:original} is Algorithm \ref{al:accelerated} with \mbox{$\mu=\epsilon=1$}. Thus, 
    $
    \beta_1 = \max(1, \blue{|}1-2\alpha\blue{|})
    $ for Algorithm \ref{al:original}. 
    Under the assumption that no integral term is present, \mbox{$\beta_1 = \blue{|}1 - 2\alpha\blue{|}$}; otherwise, $\beta_1 = 1$. On the other hand, under the constraints on $\mu, \alpha, \epsilon$ in Proposition \ref{prop:ges}, we have that $\blue{\Big|}\frac{\mu}{2} + \sqrt{\frac{\mu^2}{4} + (1 - \mu)}\blue{\Big|} < 1,$ and also \mbox{$\blue{\Big|}\frac{\mu(1 - 2\epsilon\alpha)}{2} + \sqrt{\frac{\mu^2}{4}(1-\epsilon\alpha)^2 + (1 - \mu)}\blue{\Big|} < \blue{|}1 - 2\alpha\blue{|}.$}
    Henceforth, \mbox{$\beta_2 < \beta_1$}.
\end{proof}

Proposition \ref{prop:ges} establishes that the slow and fast dynamics of Algorithm \ref{al:accelerated} must evolve at different time-scales to ensure convergence. Therefore, the convergence rate of the auxiliary variables determines the convergence speed of the optimum estimates. Thanks to adding momentum, $\gamma, \rho$ and $\alpha$ can be increased without violating the conditions of time-scale decoupling of singularly perturbed systems, improving the convergence speed.

\section{Illustrative example}\label{sec:experiments}
We compare our algorithm (\textsf{A2DMM-GT}) with the one in \cite{carnevale2023admm} (\textsf{ADMM-GT}), \blue{the accelerated gradient tracking approaches in \cite{xin2019distributed} (\textsf{AGT}) and \cite{nguyen2023accelerated} (\textsf{A2GT}),} and the older distributed gradient tracking protocol in \cite{nedic2017achieving} (\textsf{DIGing}). \blue{We consider two settings. First, a quadratic programming setup, that arises in robotic, smart grid or feature selection applications \cite{ubl2019totally}}. We set $\mathsf{N}=200$, $n=2$ and $f_i(\mathbf{x}) = \mathbf{x}^{\top}\mathbf{B}_i\mathbf{x} + \mathbf{a}_i \mathbf{x}$, where $\mathbf{B}_i$ is a positive definite matrix with uniformly random eigenvalues in the range $[1, 5]$ and $\mathbf{a}_i$ is a random vector with elements uniformly drawn in the range $[-10, 20]$. \blue{Note that $f_i(\mathbf{x})$ fulfills the assumptions in terms of strong convexity and Lipschitzness of the gradients.} Graph $\mathcal{G}$ is generated as a sparse \blue{undirected connected} proximity graph, where the nodes are randomly placed in a $2\times2$ square if and only if the new node is within $[0.1, 0.17]$ from an already placed node. \blue{The parameters
for the all the approaches are hand-tuned to achieve fastest convergence possible:} $\gamma=1.6$, $\alpha = 0.9924$, $\rho=3.028$, $\lambda=0.2$, $\mu = 1.6$, $\epsilon=0.6$ for \textsf{A2DMM-GT}; $\gamma=0.4865$, $\alpha = 0.8924$, $\rho=0.528$ for \textsf{ADMM-GT}; \blue{$\alpha=0.001$, $\beta = 0.9$ for \textsf{AGT} (we use the notation in \cite{xin2019distributed}); $\alpha=0.0008$, $\beta = 0.7$, $\gamma=0.2$ for \textsf{A2GT} (we use the notation in \cite{nguyen2023accelerated})}; $\gamma=0.0127$ for \textsf{DIGing}. We simulate $500$ steps. Convergence is assessed using the evolution of the error $\mathbf{e}^t = ||\mathbf{x}^t - \mathbf{1}_{\mathsf{N}, n} \mathbf{x}^*||$. 

\blue{The second setup is a logistic regression scenario where nodes cooperate to train a linear classifier. In this case, \mbox{$\mathsf{N}=50$}, where each node has a local dataset of $m=10$ one-dimensional points $p_i^j$ drawn from a normal distribution of mean $0$ and standard deviation $1$. The labels are also randomly drawn, such that $l_i^j \in \{-1, 1\}$. We define $\mathbf{x} = [q_1, q_2]$ and $f_i(\mathbf{x}) = \sum_{j=1}^m \log(1 + \exp^{-l_i^j(q_1 p_i^j + q_2)}) + C ||\mathbf{x}||^2$, with $C=1$ to ensure strong-convexity. The parameters for all methods are the same that in quadratic programming setup.}

\begin{figure}
    \centering
    \begin{tabular}{cc}
         
    \includegraphics[width=0.45\columnwidth]{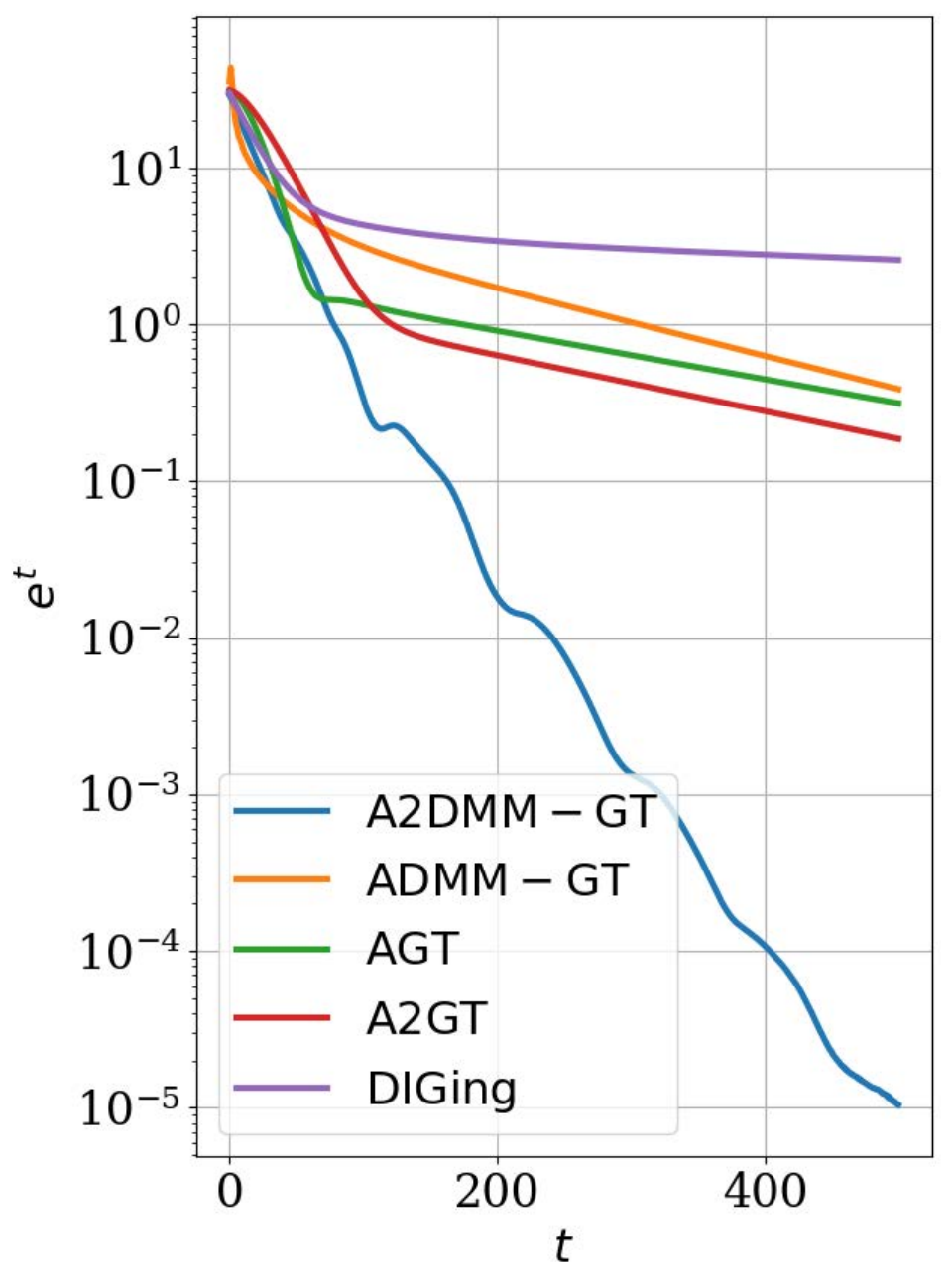} 
         &
         
    \includegraphics[width=0.45\columnwidth]{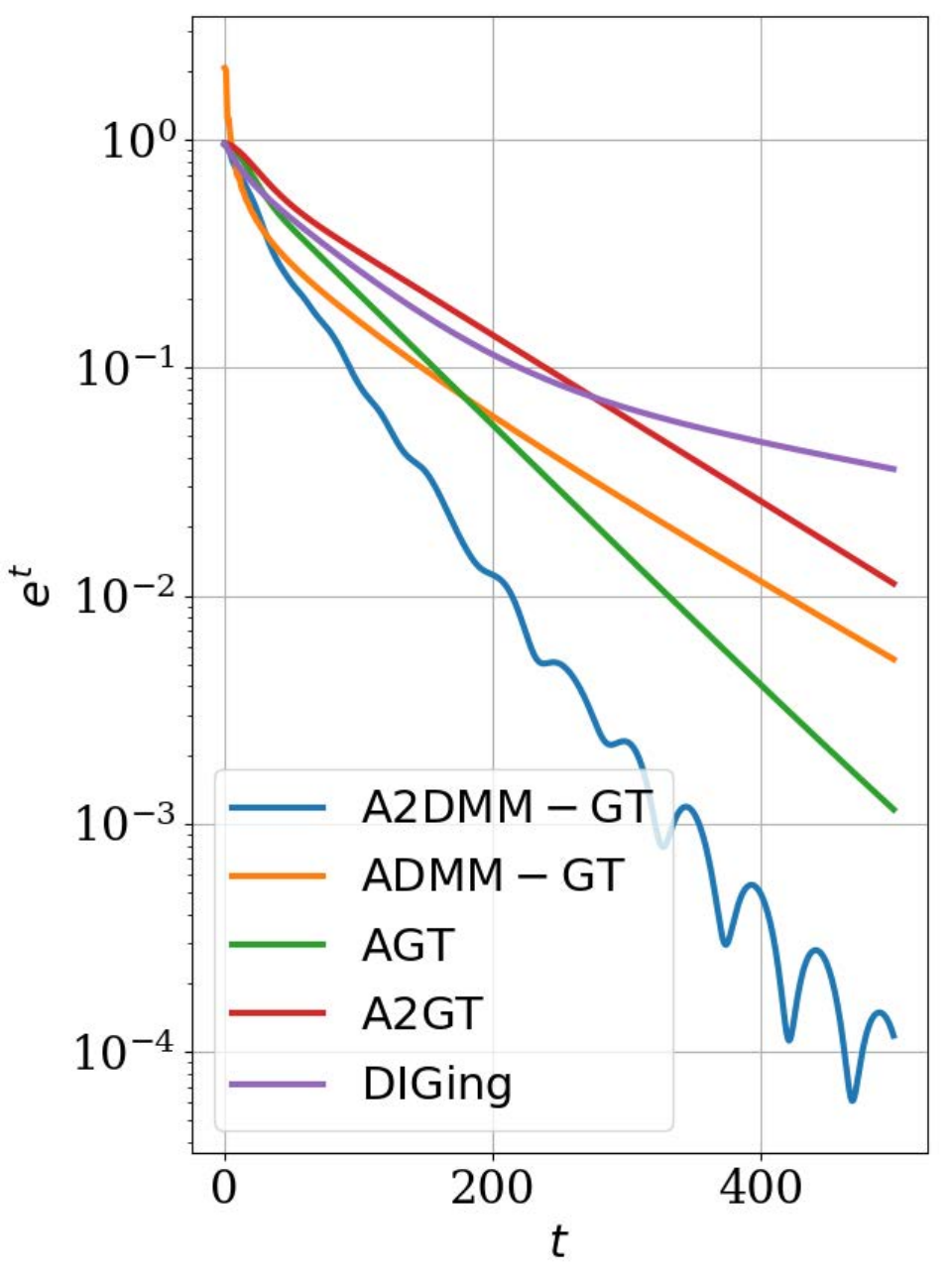} 
    \end{tabular}
    \caption{Evolution of the error with time for the five protocols under comparison. \blue{Left is the quadratic programming setup, right is the linear classification setup.}}
    \label{fig:error}
\end{figure}

The evolution of the error (Fig.~\ref{fig:error}) shows that \textsf{A2DMM-GT} \blue{is the fastest method among all the approaches for both numerical settings, with different network configurations and local costs.} Importantly, this is achieved without increasing the number of gradient computations per time step nor the communication burden, leading to a significant reduction in power consumption and total bandwidth usage in practical applications. \blue{This comes at the cost of increasing the memory burden compared to the non-accelerated methods. However, given current advances in MEMS technology, it is generally better to leverage memory against communication burden.}

\section{Conclusions}\label{sec:conclusions}
This letter presents a novel accelerated distributed optimization method for the problem of consensus optimization. The method combines the benefits of momentum to those from gradient tracking and ADMM algorithms, namely: fast convergence, simplicity of computation and robustness. From the theoretical point of view, we prove that, by adding momentum, we overcome the limitation on the convergence rate of the auxiliary variables in current ADMM-GT algorithms. From a practical perspective, we found that our algorithm is faster than the existing distributed first-order algorithms, while preserving their lightweight properties in terms of computation and communication burden. \blue{As future work, we aim at exploring the effect of adding momentum in the step size $\gamma$, the extension of the algorithm to time-varying directed graphs, and the extension to constrained settings.}

\bibliographystyle{IEEEtran}
\bibliography{biblio}

\end{document}